%% file: main.tex
\documentclass{article}
\usepackage[utf8]{inputenc}
\usepackage{amsmath}
\usepackage{amssymb}
\usepackage{amsthm}
\usepackage{graphicx}
\usepackage{algpseudocode}
\usepackage{algorithm}
\usepackage{caption}
\usepackage{subcaption}

\newtheorem{theorem}{Theorem}
\newtheorem{lemma}{Lemma}
\newtheorem{definition}{Definition}

\title{A time adaptive multirate Quasi-Newton waveform iteration for coupled problems}
\author{Niklas Kotarsky, Philipp Birken\\
Lund University, Centre for Mathematical Sciences,\\Box 118, 22100 Lund, Sweden}

\begin{document}
	
\maketitle

\begin{abstract}
We consider waveform iterations for dynamical coupled problems, or more specifically, PDEs that interact through a lower dimensional interface. We want to allow for the reuse of existing codes for the subproblems, called a partitioned approach. To improve computational efficiency, different and adaptive time steps in the subsolvers are advisable. Using so called waveform iterations in combination with relaxation, this has been achieved for heat transfer problems earlier. Alternatively, one can use a black box method like Quasi-Newton to improve the convergence behaviour. These methods have recently been combined with waveform iterations for fixed time steps. Here, we suggest an extension of the Quasi-Newton method to the time adaptive setting and analyze its properties. 

We compare the proposed Quasi-Newton method with state of the art solvers on a heat transfer test case, and a complex mechanical Fluid-Structure interaction case, demonstrating the methods efficiency. 
\end{abstract}

\input{introduction.tex}
\input{waveformIterations.tex}

\input{QuasiNewton.tex}
\input{heatTestCase.tex}
\input{FSITestCase.tex}

\section{Conclusion}

We presented an extension of the Quasi-Newton waveform relaxation method to time adaptivity. Our method uses a fixed auxiliary time grid. This introduces an interpolation error that can be controlled through the choice of the grid. Benefits are in convergence behaviour and ease of implementation. This only requires adding one extra interpolation step and we presented our implementation in the open source coupling library preCICE.

We demonstrated the methods efficiency by comparing it with optimal relaxation and QNWR using a fixed time grid on a simple thermal transfer case.  We also showcase the method on a more advanced Fluid-structure interaction test case, showing it's robustness. 

\bibliography{bib.bib}

\end{document}

%% file: introduction.tex
\section{Introduction}

We consider time dependent multiphysics aplications, or more specifically coupled PDEs that interact through a lower dimensional interface. Examples here are gas-quenching (heat), flutter in airplanes (forces) or climate models (both). To simulate these problems, it is often desirable to reuse existing codes for the different models, since they represent long term development work. This is called a partitioned approach. To increase computational efficiency we want a partitioned method that is high order, allows for different and adaptive time steps in the separate models, is robust to variations in discretization and model parameters, and contains fast inner solvers.

A major candidate to fulfill our requirements is the so called Waveform Iteration or Waveform relaxation (WR). This class of methods was developed to simulate large electrical circuits effectively in parallel \cite{WhSa86}. To this end, the underlying large system of ODEs is split into smaller ODE systems, which are solved separately given approximations of the solutions of the other subsystems. This way, an iterative process is defined. This framework naturally allows the use of different time integration methods in each subsystem. Using interpolation, different time steps can be employed in a straight forward manner. 

Convergence of waveform iterations has been studied in various continuous and discrete settings for ODEs, see for example \cite{Ne89Cont, Ne89Disc, WhSa86}. The time adaptive case, where the time grids in the sub solver change in each iteration, was analyzed in \cite{BeZe93}. WR methods were also extended to PDEs, see e.g. \cite{JaVa96a, JaVa96b,gakwma:16, BiMo19}. Unlike the ODE case, waveform iterations are not guaranteed to converge for the PDE case. 

To speed up convergence, an acceleration step is often applied. The simplest acceleration method is relaxation, which has been used in \cite{BiMo18} to achieve fast convergence for heat transfer problems. However, to obtain fast linear or even superlinear convergence, the relaxation parameter has to be optimized for the specific problem, requiring extensive analysis beforehand, see e.g. \cite{BiMo19}. A time adaptive discretization has been combined with relaxation in \cite{BiMeMo23} for coupled heat equations. Alternatively, one can use a black box method like Quasi-Newton, also referred to as Anderson acceleration. These methods have been studied extensively, see e.g. \cite{polreb:21}, and for partitioned coupling in for example \cite{HaDe10, SchUe15, SchMe17}. They have also recently been extended to WR for the case with fixed but variable time steps in \cite{RuUe21}. 

In this article, we further extend the Quasi-Newton method to the time adaptive case. An inherent difficulty is that since the number of time steps changes with every iteration, the dimension of the discrete solution vector changes with every iteration as well. Thus, the Quasi-Newton method cannot be directly applied to the sequence of iterates. Furthermore, a variable dimension of the solution vector makes the implementation into a coupling library for partitioned coupling difficult. 

We therefore suggest to interpolate the iterates on a fixed auxiliary time grid. Thus, the iteration can be written in terms of a vector of fixed dimension, at the cost of an additional interpolation step. However, the implementation into a coupling library is perfectly feasible and we present ours in the open source C++ library preCICE \cite{ChDa22}. We then analyze the behaviour of the algorithm for fixed time grids and discuss how to choose the auxiliary time grid. We also demonstrate the methods efficiency by comparing it with optimal relaxation \cite{BiMeMo23} and QNWR for fixed time grids \cite{RuUe21} on a simple thermal transfer case.  We further showcase the method on a more advanced Fluid-structure interaction problem. 

In the following, we first describe Waveform relaxation, then in section 3 the Quasi-Newton method. We analyze the new method in section 4 and provide numerical experiments in sections 5 and 6. 


%% file: waveformIterations.tex
\section{Waveform iterations}

Consider the following coupled ODEs on the interval $[0,T_f]$:
\begin{align*}
	\dot{w} &= g(w, v), \quad w(0) = w_0, \\
	\dot{v} &= h(v, w), \quad v(0) = v_0, 
\end{align*}
where $w \in \mathbb{R}^{d_w}$, and $v \in \mathbb{R}^{d_v}$. We assume that $g$ and $h$ are Lipschitz continuous functions.

To define our waveform iteration, we divide the interval $[0,T_f]$ into a finite number of time windows $[T_{w-1},T_w]$. In the following, without loss of generality, we assume that there is only one time window $[0,T_f]$. The Gau\ss -Seidel or serial waveform iteration is then given by  
\begin{align*}
	\dot{w}^{k+1} &= g(w^{k+1}, v^{k}), \quad w^{k+1}(0) = w_0 ,\\
	\dot{v}^{k+1} &= h(v^{k+1}, w^{k+1}) , \quad v^{k+1}(0) = v_0, \text{ for } k=0, ...,
\end{align*}  
where the starting guess $v^0 \equiv v_0$ is commonly used. We use the termination criterion 
\begin{align}\label{Conv-crit}
	\|v^{k+1}(T_{f})-v^{k}(T_{f})\|/\|v^{k+1}(T_{f})\| \leq Tol_{WR},
\end{align} where $\|\cdot\|$ is a vector norm. 

A discrete waveform iteration is given by discretizing both sub problems in time. To this end, we make use of the notion of a time grid, defined without loss of generality on $[0,T_f]$.

\begin{definition}
	A vector ${\cal T}\in\mathbb{R}^N$ with components $t_i\in[0,T_f]$ and
	\[0=t_0<t_1<\hdots<t_{N-1}=T_f\]
	is called a time grid. 
\end{definition}

The two discretizations are assumed to use the grids ${\cal T}_w$ and ${\cal T}_v$ with $N_w$ and $N_v$ grid points, respectively. This yields the discrete solutions $\bar{w}^{k} =\{w^{k,i}\}_{i=0}^{N_w-1}$ and $\bar{v}^{k} =\{v^{k,i}\}_{i=0}^{N_v-1}$ for the grid points in ${\cal T}_w$ and ${\cal T}_v$ respectively. The two time grids are not assumed to be matching. Thus, one has to use some kind of interpolation in the discrete waveform iteration. We define the interpolation below for the general setting where we have  data vector $x \in R^{dN}$ consisting of the $d$-dimensional components $x_i$ for $i = 0,...,N-1$, where each $x_i$ belongs to a time step $t_i \in \mathcal{T}$.

\begin{definition}
We denote a continuous interpolation of a data vector $x \in R^{dN}$ on the time grid $\mathcal{T}$ by $ \mathbf{I}_\mathcal{T}(x)$. To simplify the notation, we interpret the interpolation as a map from either a time grid and a data vector, or only a data vector to a continuous function. We use the notation $\mathbf{I}_{{\cal T}}(x)(\chi): \mathbb{R}^{N} \xrightarrow{} \mathbb{R}^{Nd}$ to denote the sampling of the interpolant $\mathbf{I}_{{\cal T}}(x)$ on a time grid $\chi$.  
\end{definition}

Using interpolation, the discrete waveform iterations are obtained by discretizing 
\begin{align*}
	\dot{w}^{k+1} & = g( w^{k+1}, \mathbf{I}_{{\cal T}_v}( \bar{v}^{k})), \quad w^{k+1}(0) = w_0, \\
	\dot{v}^{k+1} & = h(v^{k+1}, \mathbf{I}_{{\cal T}_w}( \bar{w}^{k+1})) , \quad v^{k+1}(0) = v_0,
\end{align*}
with two possibly different time discretization methods in time. 

To extend the waveform iteration to the time adaptive setting, where the two time grids change in each iteration, we denote the iteration number with a superscript k. It is worth pointing out that we do not assume that the two sequences of time grids have the same number of grid points in each iteration k. Thus, the discrete solutions are be given by $\bar{w}^{k} =\{w^{k,i}\}_{i=0}^{N_w^k}$ and $\bar{v}^{k} =\{v^{k,i}\}_{i=0}^{N_v^k}$, where ${N_w^k}$ and ${N_v^k}$ denotes the number of grid points in time grids ${\cal T}_w^k$ and ${\cal T}_v^k$, respectively. The discrete waveform iteration is then given by discretizing 
\begin{align*}
	\dot{w}^{k+1} & = g( w^{k+1}, \mathbf{I}_{{\cal T}_v^k}( \bar{v}^{k})), \quad w^{k+1}(0) = w_0, \\
	\dot{v}^{k+1} & = h(v^{k+1}, \mathbf{I}_{{\cal T}_w^k}( \bar{w}^{k+1})) , \quad v^{k+1}(0) = v_0.
\end{align*}

The convergence of waveform iterations has been studied in various continuous and discrete settings for ODEs, see for example \cite{Ne89Cont,Ne89Disc}. For ODEs, waveform iterations achieve superlinear convergence with an error bound of the form
\begin{equation*}
	||e^{k}||_{[0,T_f]} \leq \frac{(CT_f)^k}{k!} ||e^0||_{[0,T_f]},
\end{equation*}
where $C>0$ is a constant depending on the Lipschitz-constant of the coupled ODE, as well as their time discretization. The norm $||.||_{[0,T_f]}$ is the supremumnorm $||e^k||_{[0,T_f]} := \sup_{t \in [0,T_f]} ||e^k(t)||$. Convergence in the time adaptive case was proven in \cite{BeZe93} under the assumption that both time grids in the two sub solvers converge. As the authors point out in \cite{BeZe93} the assumption that the time grids converge is not necessarily realistic, since the time step is given by a controller that does not have anything to do with the waveform iteration. 

\subsection{Waveform iterations for interface coupled PDEs}
Waveform iterations have also been used in the context of coupled PDEs which interact through a lower dimensional boundary $\Gamma$, see for example \cite{JaVa96a, JaVa96b, RuUe21, BiMo19, BiMeMo23, gakwma:16}. This is accomplished by splitting the two PDEs into separate systems that interact by the exchange of boundary data. This allows us to treat the two solvers as black boxes which map interface data onto other interface data. In the continuous setting, the Gau\ss-Seidel waveform iteration for PDEs can be written as
\begin{align*}
	x_1^{k+1} & = \mathcal{S}_1( x_2^{k}) \in C\left(\Gamma \times [0,T_f]\right), \\
	x_2^{k+1} & = \mathcal{S}_2(x_1^{k+1}) \in C\left(\Gamma \times [0,T_f]\right) ,
\end{align*}
where $x_1 \in C\left(\Gamma \times [0,T_f]\right)$ and $x_2 \in C\left(\Gamma \times [0,T_f]\right)$ are interface data from the first respectively second domain. $\mathcal{S}_1$ and $\mathcal{S}_2$ are unbounded Poincaré-Steklov operators that define a map from the values of a boundary condition to the values of another boundary condition. 

The discrete waveform iteration is obtained by discretizing $ \mathcal{S}_1$ and $ \mathcal{S}_2$ in space and time. Here we assume that both solvers use fixed spatial grids, since the focus is on the time integration. If the meshes in time are also fixed, the discrete Poincaré-Steklov operators can be defined as 
\begin{equation*}
	S_i: C[0,T_f]^d \xrightarrow{} \mathbb{R}^{d N_i},
\end{equation*}  
for $i = 1,2$, where $d$ denotes the size of $x_1$ and $x_2$, and $N_i$ the size of the time grid for solver $i$. The waveform iteration can then be written as
\begin{align}\label{WRfixed}
	x_2^{k+1} = S_2 \circ \mathbf{I}_{\mathcal{T}_1} \circ S_1 \left( \mathbf{I}_{\mathcal{T}_2}\left(x_2^k\right) \right) \in \mathbb{R}^{d N_2}.
\end{align}
If we use instead a time adaptive method, the discrete operators can instead be defined by the following map
\begin{equation*}
	S^{TA}_i: C[0,T_f]^d \xrightarrow{} (\mathbb{R}^N,\mathbb{R}^{d N}),
\end{equation*} taking interface data from the other solver and returning the used time grid $\mathcal{T}_i^k$, as well as the corresponding discrete interface data. This allows us to write the waveform iteration as \begin{align}\label{badTAWR}
	\left(x_2^{k+1}, \mathcal{T}_2^{k+1} \right) = S^{TA}_2 \circ \mathbf{I}_{\mathcal{T}^{k+1}_1} \circ S^{TA}_1 \left( \mathbf{I}_{\mathcal{T}^{k}_2}\left(x_2^k\right) \right) \text{ on } \mathbb{R}^{d N^{k+1}_2 \times N^{k+1}_2}.
\end{align}


It is worth noting that unlike in the ODE case, the waveform iterations given in Equation \eqref{badTAWR} and \eqref{WRfixed} are not guaranteed to converge. 
However, there exist convergence theorems for various specific cases, see for example  \cite{gakwma:16,JaVa96a, JaVa96b}.

To improve the convergence behaviour, some kind of acceleration is often employed. The simplest acceleration method is relaxation, where a weighted  average is taken of the iterates. Relaxation was used in a fully discrete time adaptive setting in \cite{BiMeMo23}, where it was proposed to sample the old interpolations to the newest time grid ${\cal T}_{2}^{k+1}$. In our notation this is 
\begin{align*}
	x_2^k = \theta \hat{x}_2^k + (1 - \theta)\mathbf{I}_{\mathcal{T}_2^k}(x_2^{k})({\cal T}_{2}^{k+1}),
\end{align*}
where $\theta \in [0,1]$ is the relaxation parameter. The convergence speed of the waveform iteration with relaxation is highly dependent on the relaxation parameter. For two coupled heat equations, the relaxation parameter has been optimized in \cite{BiMo18}, which achieves fast convergence in the time adaptive setting. Another option would be to apply a modified Quasi-Newton method.

%% file: QuasiNewton.tex
\section{Quasi-Newton method}

Quasi-Newton methods have been studied in the context of coupled problems in for example \cite{HaDe10,SchMe17,SchUe15} for the case where both solvers do one time step. They have recently been extended in \cite{RuUe21} to waveform iterations where both solvers use fixed equidistant time grids. For the case where the sub solvers use fixed but variable time steps, the residual is given by
\begin{equation}
	r(x) = \mathcal{S}_{2}  \circ \mathbf{I}_{\mathcal{T}_1} \circ \mathcal{S}_{1}( \mathbf{I}_{\mathcal{T}_2}(x)) - x \in \mathbb{R}^{d N_2},
\end{equation} 
where the term $\mathcal{S}_{2}  \circ \mathbf{I}_{\mathcal{T}_1} \circ \mathcal{S}_{1}( \mathbf{I}_{\mathcal{T}_2}(x))$ corresponds to the fixed point iteration in Equation \eqref{WRfixed}. 
To simplify the notation, we introduce the operator $H:R^{dN_2} \xrightarrow{} R^{d N_2}$ defined as 
\begin{equation*}
	H :=\mathcal{S}_{2}  \circ \mathbf{I}_{\mathcal{T}_1} \circ \mathcal{S}_{1}( \mathbf{I}_{\mathcal{T}_2}(x)), 
\end{equation*} 
simplifying the residual to 
\begin{equation*}
	r(x) = H(x) - x.
\end{equation*}

In order to avoid badly conditioned matrices, the Quasi-Newton method is often not applied to the residual $r(x) = H(x) - x$ directly. Instead it is applied to the modified residual defined as 
\begin{equation*}
	\tilde{r}(x) = H^{-1}(x) - x.
\end{equation*}
The update for the Quasi-Newton method is then given by
\begin{align}\label{QNWRFixed}
	\begin{split}
		\hat{x}^k &= H(x^k) \\
		x^{k+1} &= \hat{x}^{k} - W_k (V_k^T V_k)^{-1}V_k^{T} \tilde{r}(\hat{x}^{k}),
	\end{split}
\end{align} 
where  
\begin{equation*}
	V_k = [\tilde{r}(\hat{x}^1)-\tilde{r}(\hat{x}^0), ... , \tilde{r}(\hat{x}^k)-\tilde{r}(\hat{x}^{k-1})]  \in \mathbb{R}^{dN_2  \times k}
\end{equation*} 
and 
\begin{equation*}
	W_k = [\hat{x}^1 -\hat{x}^0, .... , \hat{x}^{k}- \hat{x}^{k-1}] \in \mathbb{R}^{dN_2 \times k}.
\end{equation*}

The modified residual $\tilde{r}(x)$ is generally not known explicitly. However, one can use the relation 
\begin{equation*}
	r(x) = H(x)-x = H^{-1} \circ H(\hat{x}) -  H^{-1}(\hat{x}) = -\tilde{r}(\hat{x}). \label{residuals}
\end{equation*} 
Using this, the matrix $V_k$ can be written as 
\begin{equation*}
	V_k = [-r(x^1)+r(x^0), ... , -r(x^k)+r(x^{k-1})]  \in R^{nd  \times k}. 
\end{equation*} 
This update requires one to have computed $x^1$, which is why one in the  first iteration commonly uses  a relaxation step.

Using the QR decomposition of $V_k = Q_k R_k$, where $Q_k$ has orthonormal columns and $R_k$ is an upper triangular matrix, the term $(V_k^TV_k)^{-1}V_k^T r(x^k)$ can efficiently be computed by solving $R_k \alpha = Q^T r(x^k)$ with backwards substitution. The update $W_k(V_k^TV_k)^{-1}V_k^T r(x^k) $ can then be computed by $W_k \alpha$. 

In \cite{RuUe21} an alternative extension of the QN method to waveform iterations, referred to as the reduced Quasi-Newton waveform relaxation (rQNWR), is also presented. There, only the last time step is used to construct the residual and the matrix $V_k$, reducing the cost of the QR decomposition. Here, we focus on the version where all time steps are included in the residual and the matrix $V_k$. This allows for stronger convergence theorems in Section \ref{analysisQN}.

\subsection{Generalization to time adaptive sub-solvers}\label{QNWRTA}

The main difficulty with generalizing the Quasi-Newton method to time adaptive waveform iterations is that the number of time steps and the dimension of the discrete solution vector vary between iterations. Thus, the Quasi-Newton method for fixed time grids in equation \eqref{QNWRFixed} cannot directly be applied to the time adaptive waveform iterations in equation \eqref{badTAWR}. 

To extend the QN method to time adaptive waveform iterations, one has to either modify the Quasi-Newton method to work with a solution vector whose dimension changes in each iteration, or alternatively modify the time adaptive waveform iterations by interpolating everything to a fixed time grid. The former is possible by interpolating all vectors in the matrices $V_k$ and $W_k$, and then sampling the interpolant on the new time grid $\mathcal{T}_k$ after every iteration. Thus, one needs dynamic memory management and one needs to adjust the core QN update. This method is difficult to implement in a coupling library, and quite invasive. Furthermore, the convergence behaviour is hard to analyze, making it difficult to justify the added complexity of the implementation.

We therefore propose to instead interpolate the solution on an auxiliary fixed time grid $\mathcal{T}_{QN} = \{t^i_{QN}\}_{i=0}^{N_{QN}-1}$. We denote an interface vector on that grid by $x_{QN}$. It has the fixed dimension $dN_{QN}$. This changes the iteration in Equation \eqref{badTAWR} to
\begin{align} \label{WRTAQN}
	x_{QN}^{k+1} = \mathbf{I}_{\mathcal{T}_2^{k+1}} \circ \mathcal{S}^{TA}_{2}  \circ \mathbf{I}_{\mathcal{T}_1^{k+1}} \circ \mathcal{S}^{TA}_{1}\left( \mathbf{I}_{\mathcal{T}_{QN}}\left(x_{QN}^{k}\right)\right)\left(\mathcal{T}_{QN}\right).
\end{align} 

The Quasi-Newton update can then be computed exactly as described for fixed time steps, yielding the following time adaptive QNWR method:
 \begin{align}\label{timeAdaptiveQN}
	\begin{split}
		\hat{x}_{QN}^{k+1} &= \mathbf{I}_{\mathcal{T}_2^{k+1}} \circ \mathcal{S}^{TA}_{2}  \circ \mathbf{I}_{\mathcal{T}_1^{k+1}} \circ \mathcal{S}^{TA}_{1}\left( \mathbf{I}_{\mathcal{T}_{QN}^k}\left(x_{QN}^{k}\right)\right)\left(\mathcal{T}_{QN}\right) \\
		x_{QN}^{k+1} &= \hat{x}_{QN}^{k+1} + W_k (V_k^T V_k)^{-1}V_k^{T} r_{QN}(x_ {QN}^k),
	\end{split}
\end{align}
where the residual is given by 
\begin{align*}
	r_{QN}(x^k) = \hat{x}_{QN}^{k+1} - x^{k}_{QN} \in \mathbb{R}^{d N_{QN}},
\end{align*} 
and we have the matrices 
\begin{align*}
	V^{QN}_k &= [r_{QN}(x_{QN}^1)-r_{QN}(x_{QN}^0), ... , r_{QN}(x_{QN}^k)-r_{QN}(x_{QN}^{k-1})]  \in R^{n_{QN} d  \times k},\\
	W^{QN}_k &= [\hat{x}_{QN}^1 -\hat{x}_{QN}^0, .... , \hat{x}_{QN}^{k}- \hat{x}_{QN}^{k-1}] \in R^{d N_{QN}  \times k}.
\end{align*}

We have thus re-used the Quasi-Newton algorithm from the fixed time step case by adding one extra interpolation step, making it feasible to implement in coupling codes like preCICE. If both sub solvers use constant time grids, then we recover the QNWR method for fixed time steps if we choose $\mathcal{T}_{QN} = \mathcal{T}_2$. The interpolation error introduced by interpolating the iterates from $\mathcal{T}_2$ to $\mathcal{T}_{QN}$ when $\mathcal{T}_{QN} \neq \mathcal{T}_2$ will be investigated in section \ref{analysisQN}.

\subsection{Implementation details of time adaptive QNWR in the open source coupling library preCICE}\label{implementationDetail}

preCICE is an open-source couping library designed to couple different solvers together in a minimally invasive fashion. For a more thorough description we refer to \cite{ChDa22}. In essence, preCICE controls the coupling between the two solvers, handling everything from the communication between the solvers to interpolating the interface data, as well as checking the termination criterion. For our purposes it is enough to know that preCICE has knowledge of the sub solvers interface data, time grids and the interpolations. Internally in preCICE the interface data, time steps and the interpolants are all stored in a class which all acceleration methods have full access to. 

preCICE also supports a wide range of different acceleration methods such as relaxation, aitken and QN. Furthermore, preCICE also supports a wide range of different options for such as scaling and filtering, as well as the QN extension IQN-IMVJ \cite{SchUe15}. The QN implementation in preCICE uses a vector to internally store the interface data. Furthermore, the implementation uses a class structure, with a base class for the QN methods containing common functionality.

We have implemented the time adaptive QNWR extension from section \ref{QNWRTA} in preCICE, scheduled to be released in version 3.2.0. For this, not many changes needed to be made. First, the pre-processing step, where the interface data is assembled in a large vector, has to be extended to interpolate the waveform iteration on the QN time grid $\mathcal{T}_{QN}$. Second, the post-processing step, where the QN vector is split back into interface data, has to be extended to also update the interpolant. These changes require that the QN methods can sample, reset and write new data to the interpolation object. However, they only require modifications to the QN base class. Lastly, in order to be able to select the Quasi-Newton time grid $\mathcal{T}_{QN}$ as the time grid of the second sub solver $\mathcal{T}_2$ the data structures in the QN classes have to be initialized after the first waveform iteration. This choice requires small changes to the inherited classes of the QN base class and is further discussed in Section \ref{preliminaryStudy}. 

\subsection{Linear analysis of the new method}\label{analysisQN}

To analyze the interpolation error in the new method \eqref{timeAdaptiveQN}, we assume that the solvers $S_1$ and $S_2$ both use the fixed time grid $\mathcal{T}_2$. 
We further simplify the analysis by assuming that the solvers $S_1$ and $S_2$ are of the form, 
\begin{align*}
	x_1 &= A_1 x_2 + b_1 \\
	x_2 &= A_2 x_1 + b_2,
\end{align*} 
where $A_1, A_2 \in \mathbb{R}^{d N_2 \times d N_2}$ and $b_1, b_2 \in \mathbb{R}^{d N_2}$. The fixed point equation of the coupled system, given by 
\begin{align*}
	x_2 = S_2 \circ S_1(x_2),
\end{align*}
 simplifies then to 
 \begin{equation}\label{simWRTA}
	x_2 = A x_2 + b,
\end{equation} 
where $A \in \mathbb{R}^{d N_2 \times d N_2}$ and $b \in \mathbb{R}^{d N_2}$. 

In addition, we only consider linear interpolation, meaning that we can denote $I_{\mathcal{T}_{QN}}(x_{QN})({\cal T}_S)$ and $I_{\mathcal{T}_{S}}(x_{2})({\cal T}_{QN})$ by $\Phi x_{QN}$ and $\Theta x_{2}$ respectively, where $\Phi \in \mathbb{R}^{d N_2\times d N_{QN}}$ and $\Theta \in \mathbb{R}^{d N_{QN}\times d N_{2}}$ are matrices. For this simplified problem, the fixed point equation for the time adaptive QN method \eqref{timeAdaptiveQN} is given by \begin{align}\label{simQNWRTA}
	\begin{split}
		x_{QN} =\Theta A \Phi x_{QN} + b, \\
	\end{split}
\end{align} 

We now assume that the matrices $A$, $A-I$ and $\Theta A \Phi  - I$ are invertible. While these assumptions can be weakened this will not yield significantly more insight into the QNWR algorithms performance. Thus, the solutions of the fixed point equations \eqref{simWRTA} and \eqref{simQNWRTA} are characterized by 
\begin{equation*}
	x^* = (A - I)^{-1}b,
\end{equation*} 
and 
\begin{equation*}
	x^*_{QN} = (\Theta A \Phi - I)^{-1} \Theta b,
\end{equation*}  
respectively. The interpolation error of the QNWR method is then
\begin{equation*}
e_{QN} = x^* -  \Phi  x_{QN}^{*}.
\end{equation*}
Under these conditions it fulfills the following lemma.

\begin{lemma} \label{interpError}
If the coupled problem is linear and constant time grids are employed, then the interpolation error of the time adaptive QNWR \eqref{timeAdaptiveQN} is given by
\begin{equation}
e_{QN} = ( \Phi \Theta - I )b + ( \Phi  \Theta  (A^{-1} - \Phi \Theta )^{-1} \Phi \Theta  -(A^{-1}-I)^{-1} )b.
\end{equation} 
\end{lemma}
\begin{proof}
	The interpolation error is given by 
	\begin{equation*}
		e_{QN} = x^* -  \Phi  x_{QN}^{*} = (A- I)^{-1}b -  \Phi ( \Theta A \Phi  - I)^{-1} \Theta b.
	\end{equation*}
	Using that $A$, $A-I$ and $\Theta A \Phi  - I$ are invertible, the Sherman-Morrison-Woodbury formula yields
	\begin{align*}
		( \Theta A \Phi  - I)^{-1} &=  - I -  \Theta (A^{-1} - \Phi  \Theta )^{-1} \Phi,\\
		(A-I)^{-1} &= -I - (A^{-1}-I)^{-1},
	\end{align*}
	resulting in the following formula for the error
	\begin{equation*}
		e_{QN} = \left(-I - (A^{-1}-I)^{-1}\right)b - \Phi \left(- I  - \Theta (A^{-1} - \Phi  \Theta )^{-1} \right) \Theta b. 
	\end{equation*} 
	Simplifying yields,
	\begin{align*}
		e_{QN} = ( \Phi \Theta - I )b + ( \Phi  \Theta  (A^{-1} - \Phi \Theta )^{-1} \Phi \Theta  -(A^{-1}-I)^{-1} )b. 
	\end{align*} 
\end{proof}

It is worth pointing out that this interpolation error is zero if $\Phi \Theta = I$, which is the case if $\mathcal{T}_{QN} = \mathcal{T}_{2}$. If  $\Phi \Theta \neq I$, then the QNWR algorithm gives rise to an interpolation error which affects the accuracy of the simulation. For fixed time grids, choosing $\mathcal{T}_{QN}$ as the first time grid of the second sub solver $\mathcal{T}_2^1$ guarantees a zero interpolation error. For further investigations we refer to our experiments.  

Regarding convergence, general theorems are difficult, but a great deal can be said for linear fixed point equations like \eqref{simQNWRTA}. In \cite{HaDe10}, the following relationship was derived between the iterates $x^{k}_{GM}$ of GMRES and $x^{k}_{QN}$ of QN acceleration: 
\begin{align*}
	x^{k+1}_{QN} = A x^{k}_{GM} + b.
\end{align*} 
This relationship immediately proves the following theorem, since GMRES finds the exact solution in $d$ iterations.

\begin{theorem}\label{FiniteConvQN}
For a linear coupled system, QN converges to the exact solution in $d+1$ iterations, where d corresponds to the number of interface unknowns.
\end{theorem}

Thus, Theorem \ref{FiniteConvQN} guarantees that the time adaptive QNWR method converges to $x^*_{QN}$ in a finite number of iterations. This means that the choice of a fixed auxiliary grid when defining the new method gives us a convergent method, with a limit that contains an interpolation error, which can be controlled throught the choice of the auxiliary grid. 

As a final pont, we discuss termination: The algorithm might not terminate, since the residual is also affected by the interpolation error through the relation
\begin{equation*}
	r(\Phi x_{QN}^*) = r(\Phi x_{QN}^*)-r(x^*) = (A-I)(\Phi x_{QN}^*-x^*) = (A - I ) e_{QN}.
\end{equation*}

The residual of the last time step, which is used in the termination criterion of the waveform iteration \eqref{Conv-crit}, is given by
\begin{equation}\label{res}
\hat{r}(\Phi x_{QN})(t_f) = G \left( \left(A -I\right) \Phi x_{QN} - b \right),
\end{equation}
where $G = \begin{bmatrix}
	0 & \dots & 0 & I
\end{bmatrix} \in \mathbb{R}^{Nd \times d}$. We have the following result:

\begin{theorem}\label{QNWRTATerm}
  If the coupled problem is linear and the time grids $\mathcal{T}_2$ and $\mathcal{T}_{QN}$ are fixed, but possibly different, then $\hat{r}(\Phi x_{QN}^*)(t_f)=0$.
\end{theorem}

\begin{proof}
To show that the rhs of \eqref{res} is 0 for $x_{QN}^{*}$, we use that $t_f$ belongs to the time grids ${\cal T}_{QN}$ and ${\cal T}_2$. Thus, the interpolation matrices $\Phi$ and $\Theta$ have the following structure
\begin{equation*}
	\begin{bmatrix}
		\ddots &  & &   \\
		& \ddots & \vdots&   \\
		& \dots & \ddots &  \\
		0 & \dots & 0 & I
	\end{bmatrix},
\end{equation*}
where the last identity block has dimension $d$. 
	
The solution of the fixed point iteration in Equation \eqref{simQNWRTA} fulfills
\begin{align*}
	0 = ((\Theta A \Phi - I) x^*_{QN} + \Phi b).
\end{align*}
This holds in particular for the last row where we have
\begin{align*}
	0 = G \left( \left(\Theta A \Phi - I \right) x^*_{QN} - \Theta b \right).
\end{align*}
Simplifying using the fact that $G \Theta = G$ and $G \Phi = G$ yields
\begin{align*}:
0 =  \left( \left(G\Theta A \Phi - G \right) x_{QN}^* - G \Theta b \right) = G \left( (A-I)\Phi x_{QN}^* - b \right).
\end{align*}
Thus, the residual in the last time step is $0$ for $x^{*}_{QN}$. 
\end{proof}

Thus, in combination with Theorem \ref{FiniteConvQN}, this theorem guarantees that the time adaptive QNWR terminates after a finite number of iterations independently of the chosen time grid ${\cal T}_{QN}$.

%% file: heatTestCase.tex
\section{Heat test case}

To test the new method, we first consider two coupled linear heat equations as in \cite{BiMeMo23,BiMo18, BiMo19}. The combined problem is given by
\begin{align*}
	\alpha_1 \dot{u}_1(t,x) - \lambda_1 \Delta u_1(t,x) &= 0,  \quad x \in \Omega_1, \\
	\alpha_2 \dot{u}_2(t,x) - \lambda_2 \Delta u_2(t,x) &= 0, \quad  x \in \Omega_2, \\
	u_1(x,t) &= u_2(x,t), \quad  x \in \Gamma, \\
	u_2(x,0) &=  u_{0_2}(x), \quad  x\in \Omega_2,\\
	u_1(x,t) &= 0, \quad  x\in \partial \Omega_1 \setminus \Gamma, \\
	u_2(x,t) & = 0, \quad  x\in \partial \Omega_2 \setminus \Gamma, \\
	u_1(x,0) &=  u_{0_1}(x), \quad  x\in \Omega_1,\\
	\lambda_2 \frac{\partial u_2(x,t)}{\partial n} &= \lambda_1 \frac{\partial u_1(x,t)}{\partial n}, \quad  x \in \Gamma, 
\end{align*}
where $t\in[0,T_f]$, $\Omega_1 = [-1,0]\times[0,1]$, $\Omega_2 = [0,1]\times[0,1]$ and $\Gamma=\Omega_1\cap \Omega_2$. The constants $\lambda_1$ and $\lambda_2$ denote the thermal conductivities for the different materials in $\Omega_1$ and $\Omega_2$. Likewise, $\alpha_{m}$ for $m=1,2$ are constants defined by $\alpha_m = c_m \rho_m$, where $\rho_m$ is the density and $c_m$ the specific heat capacity of the material. The different material parameters used in the study are shown in Table~\ref{materials}. The initial value is given by $u(x,t) = 500\sin(\pi y)\sin(\pi(x + 1)/2)$.
\begin{table}
	\begin{center}
		\begin{tabular}{ |c|c|c| } 
			\hline
			Material & $\alpha$ $[J/(Km^3)]$ & $\lambda$ $[W/(mK)]$ \\ \hline
			Steel & 3 471 348 & 49 \\ \hline
			Water & 4 190 842 & 0.58 \\ \hline 
			Air & 1 299 465 & 0.024 \\
			\hline
		\end{tabular}
	\end{center}
	\caption{The different material parameters.}
	\label{materials}
\end{table} 

This coupled problem will be solved using Dirichlet-Neumann waveform iterations. Here, we use the Gau\ss -Seidel variant. First, the continuous Dirichlet problem on the time window $[0,T_f]$ is given by

\begin{align*}
	\alpha_1 \dot{u}^{k+1}_1(t,x) - \lambda_1 \Delta u^{k+1}_1(t,x) &= 0,  \quad x \in \Omega_1, \\
	u^{k+1}_1(x,t) &= 0, \quad  x\in \partial \Omega_1 \setminus \Gamma, \\
	u^{k+1}_1(x,t) &= u^k_\Gamma, \quad  x \in \Gamma,  \\
	u_1^{k+1}(x,0) &=  u_{0_1}(x), \quad  x\in \Omega_1,
\end{align*}
where $u_\Gamma^k$ denotes the interface temperature. Second, the Neumann problem on $[0,T_f]$ is given similarly by
\begin{align*}
	\alpha_2 \dot{u}^{k+1}_2(t,x) - \lambda_2 \Delta u^{k+1}_2(t,x) &= 0, \quad  x \in \Omega_2 \\
	u^{k+1}_2(x,t) &= 0, \quad  x\in \partial \Omega_2 \setminus \Gamma \\
	\lambda_2 \frac{\partial u^{k+1}_2(x,t)}{\partial n} &= q^{k+1}, \quad  x \in \Gamma, \\
	u_2^{k+1}(x,0) &=  u_{0_2}(x), \quad  x\in \Omega_2, 
\end{align*}
where $q^{k+1}=\lambda_1 \frac{\partial u_1^{k+1}(x,t)}{\partial n}$ denotes the heat flux from the Dirichlet domain. 

Exactly as in \cite{BiMo18} and \cite{BiMeMo23}, the two sub-problems are discretized using linear finite elements in space and SDIRK2 in time. For the fixed time step methods, the time step size is chosen such that both solvers obtain the same CFL number. This corresponds to choosing 
\begin{align*}
	\Delta t_1 &= T_f/(max(1, \lfloor(\alpha_1 \lambda_2)/(\alpha_2 \lambda_1)\rfloor)N), \\
	\Delta t_2 &= T_f/(max(1, \lfloor(\alpha_2 \lambda_1)/(\alpha_1 \lambda_2)\rfloor) N,
\end{align*} where the integer $N$ corresponds to a number of base time steps.

For the time adaptive methods, where the sub solvers select the time step based on an error measurement, a lower accuracy solution $\hat{u}$ can be obtained from the SDIRK2 method. The local error can then be estimated by $l^{k+1, n+1}_u = \tilde{u}^{k+1, n+1} - u^{k+1, n+1}$. The next time step is given by a PI controller, defined as
\begin{equation*}
	\Delta t^{k+1, n+1} = \Delta t^{k+1, n} \left (\frac{TOL_{TA}}{||l^{n+1}||_2}\right )^{1/12}\left (\frac{TOL_{TA}}{||l^{n}||_2}\right)^{1/12}.
\end{equation*} 
For the first time step we set $l^{0} = TOL_{TA}$, since we do not have any information of the previous error estimate. The first time step size is given by
\begin{equation*}
	\Delta t^0 = \frac{|T_f-T_0|TOL_{TA}^{1/2}}{100(1+||f(u_0)||_2)},
\end{equation*}
as suggested in \cite{So}. We choose $TOL_{TA}=TOL_{WR}/5$, compare \cite{BiMeMo23}. 

\subsection{Preliminary experiments}\label{preliminaryStudy}

We now study the choice of the time grid $\mathcal{T}_{QN}$ in the QN method for the case were both sub solvers use equidistant time grids, as well as when both sub solvers are time adaptive. For these experiments both the test case and algorithm will be implemented using python based on \cite{Me25}, allowing us to easily change the time grid $\mathcal{T}_{QN}$ inside the QN method. The updated code with the QN implementation is available under \cite{code}.

\subsubsection{Interpolation error and termination}

We first consider equidistant time grids and use $N = 4, ..., 256$ number of base time steps. We also use an equidistant time grid in the Quasi-Newton method with grid sizes $N_{QN} = 10, 100, 10000$. The tolerance for the termination criteria was set to $10^{-12}$ in combination with a maximal number of QN iterations of 20. The error in the last time step and the number of iterations are reported in Figure \ref{iterationsQNF} and \ref{ErrorQNF}, respectively, for the different test cases. 

\begin{figure}
	\centering
	\begin{subfigure}[b]{0.3\textwidth}
		\centering
		\includegraphics[width=\textwidth]{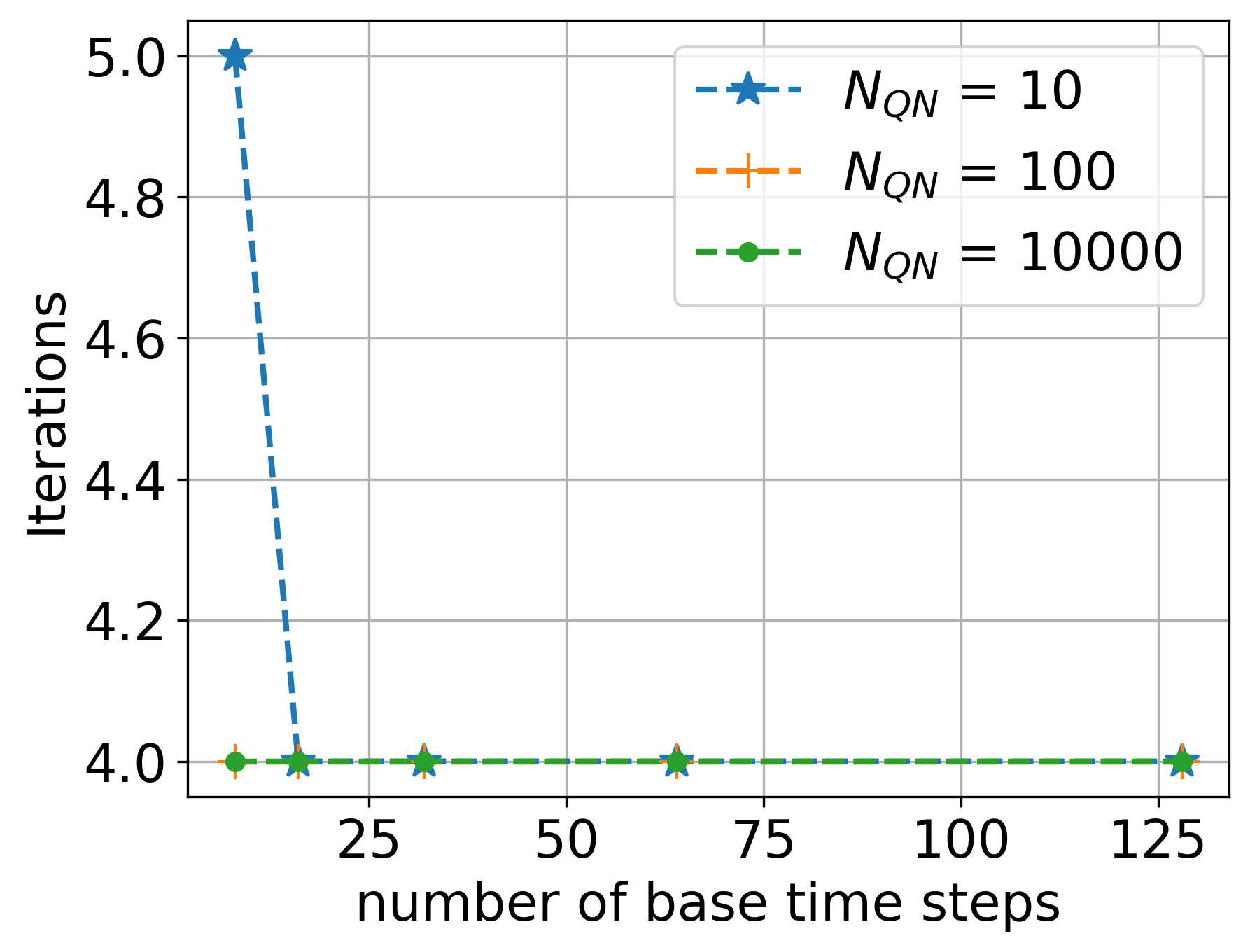}
		\caption{Air-steel}
	\end{subfigure}
	\hfill
	\begin{subfigure}[b]{0.3\textwidth}
		\centering
		\includegraphics[width=\textwidth]{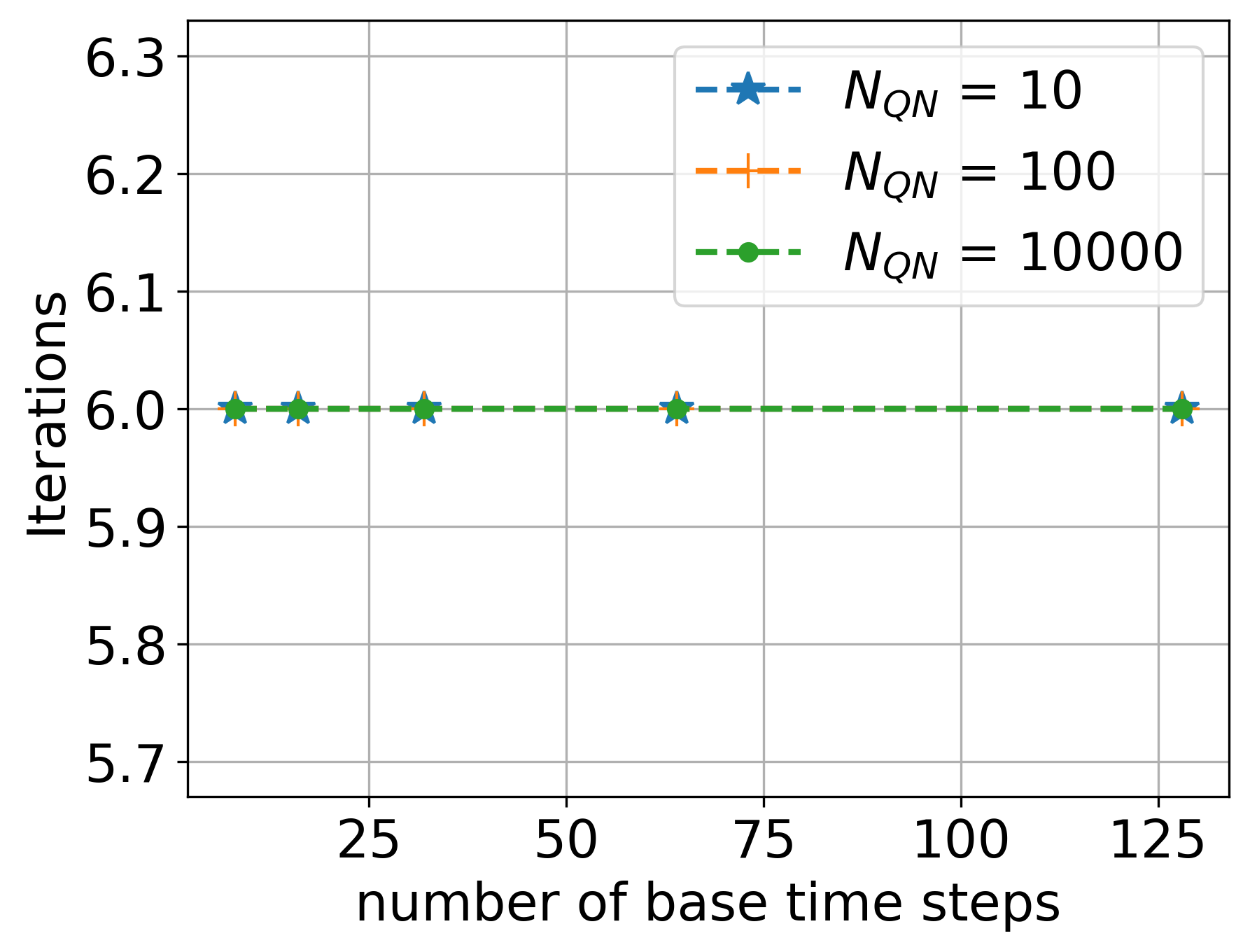}
		\caption{Air-water}
	\end{subfigure}
	\hfill
	\begin{subfigure}[b]{0.3\textwidth}
		\centering
		\includegraphics[width=\textwidth]{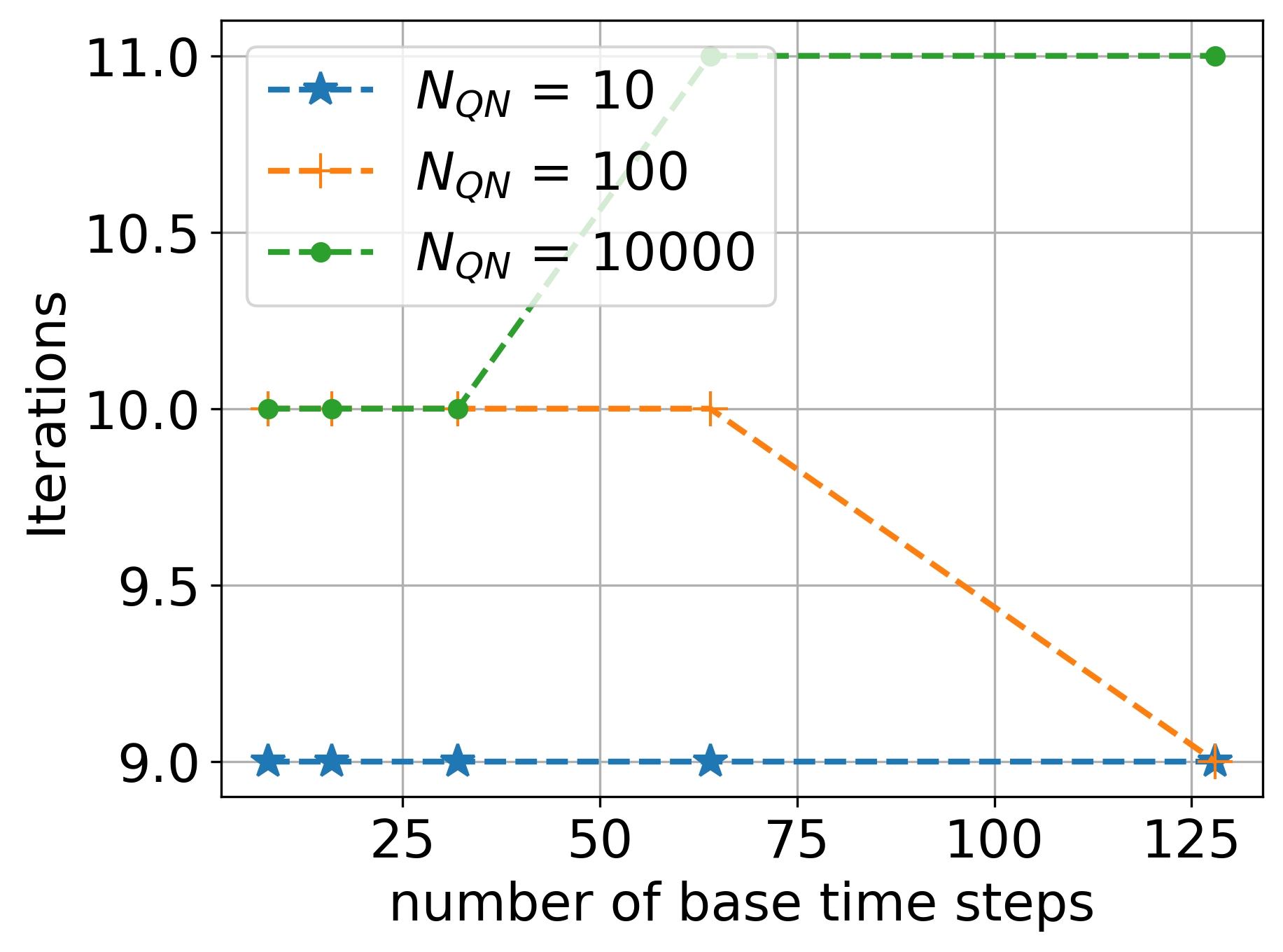}
		\caption{Water-steel}
	\end{subfigure}
	\hfill
	\caption{Number of iteration compared to the number of base time steps for different number of QN time steps.}
	\label{iterationsQNF}
\end{figure}

\begin{figure}
	\centering
	\begin{subfigure}[b]{0.3\textwidth}
		\centering
		\includegraphics[width=\textwidth]{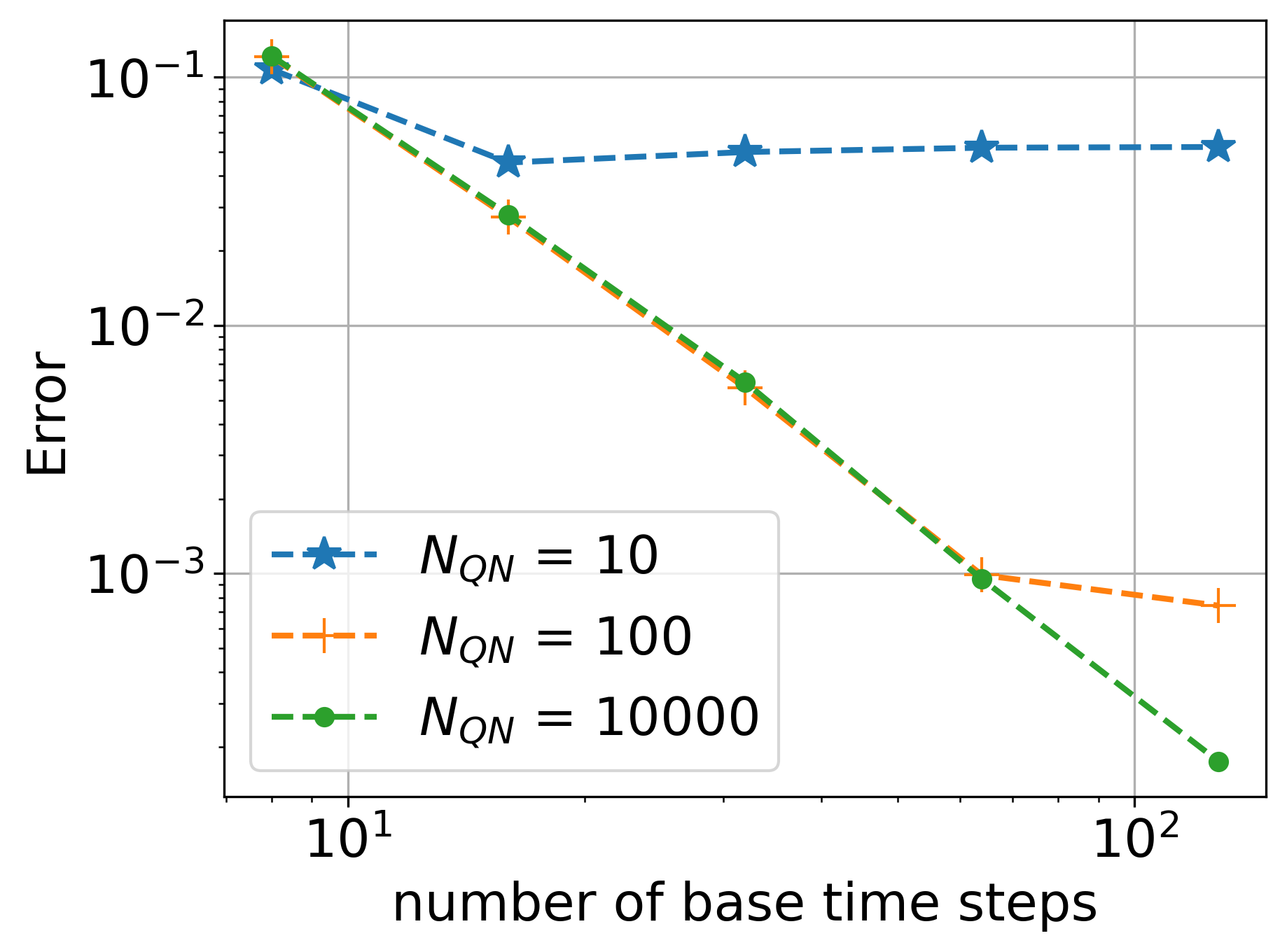}
		\caption{Air-steel}
	\end{subfigure}
	\hfill
	\begin{subfigure}[b]{0.3\textwidth}
		\centering
		\includegraphics[width=\textwidth]{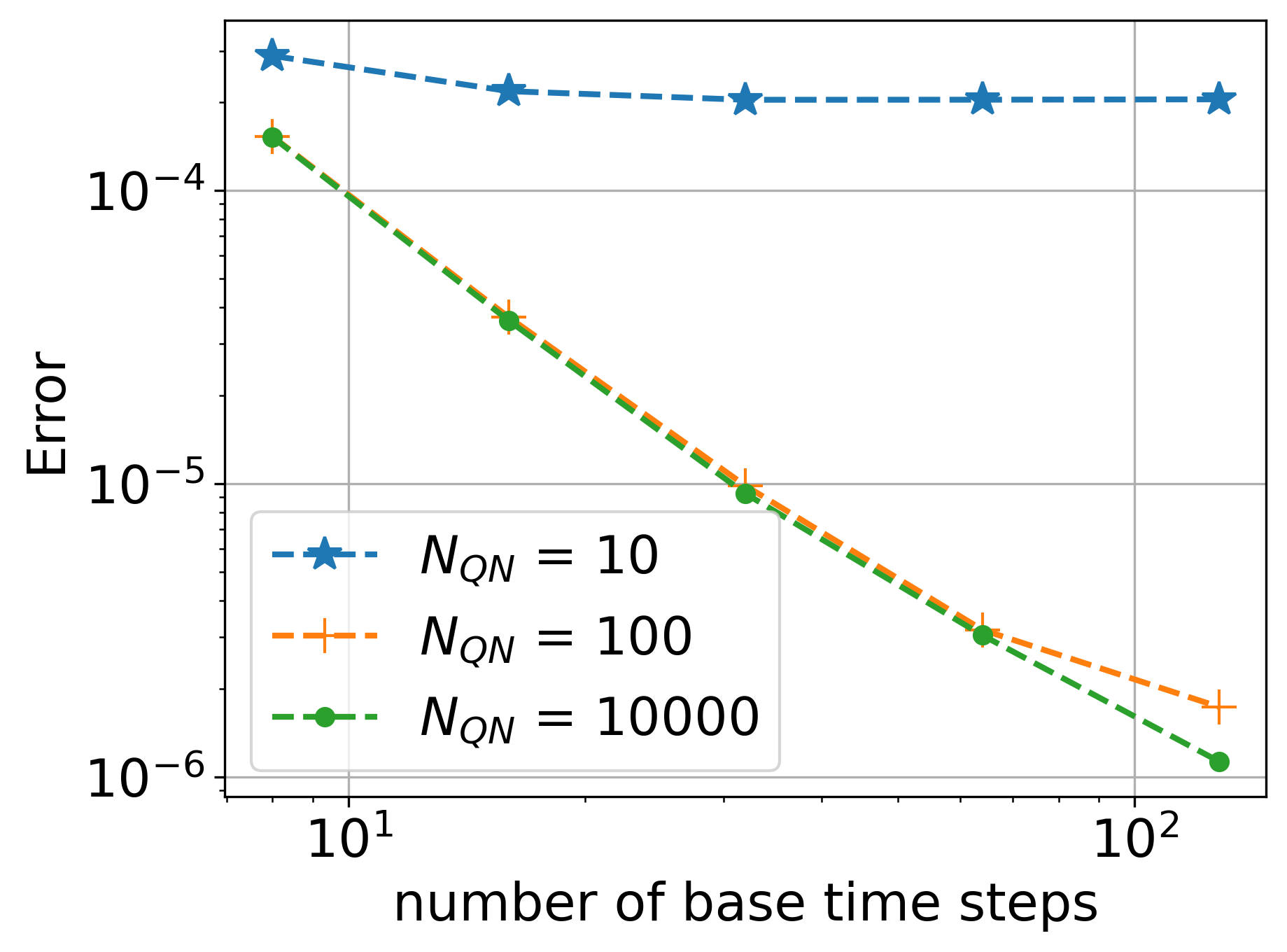}
		\caption{Air-water}
	\end{subfigure}
	\hfill
	\begin{subfigure}[b]{0.3\textwidth}
		\centering
		\includegraphics[width=\textwidth]{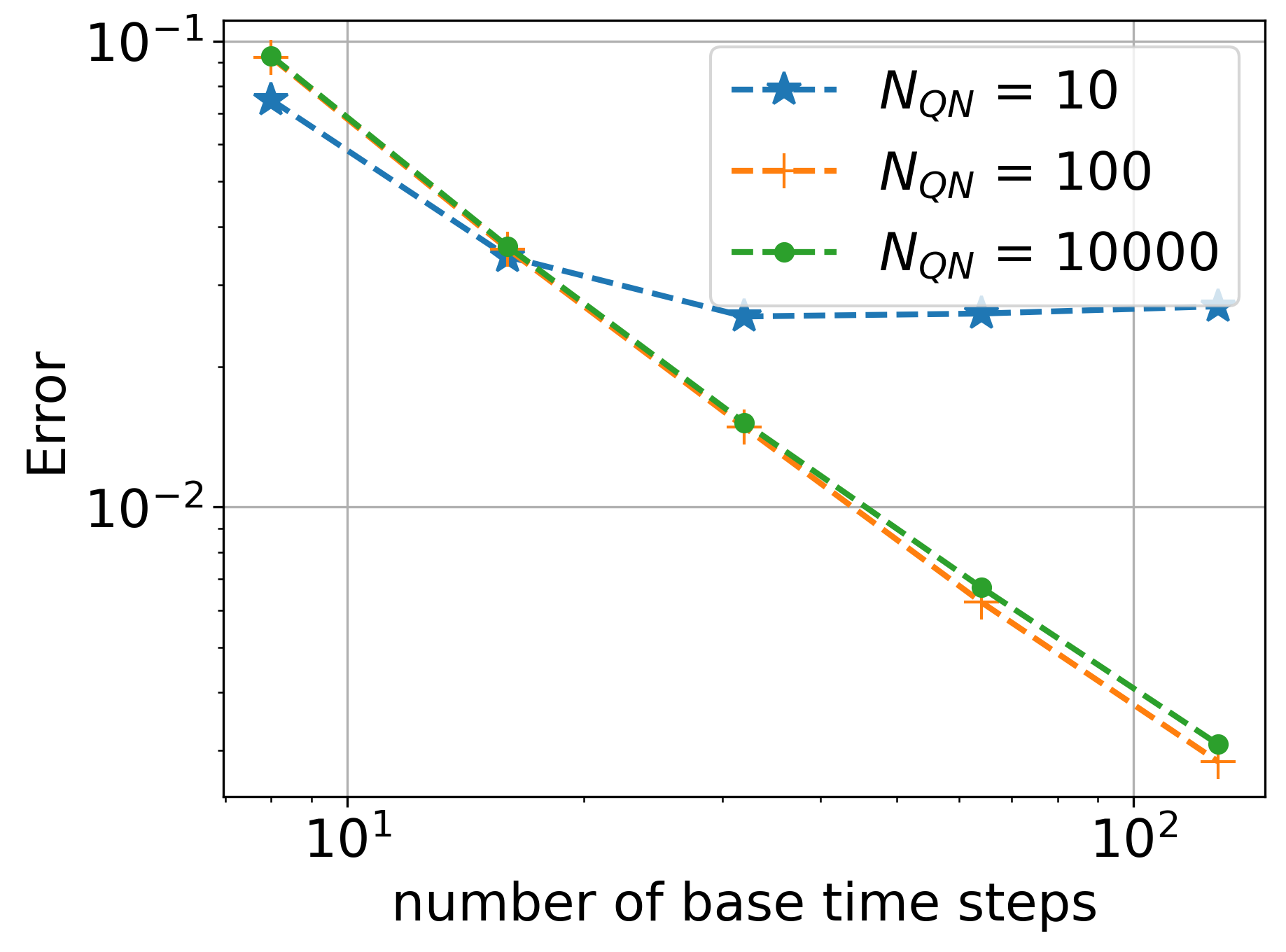}
		\caption{Water-steel}
	\end{subfigure}
	\hfill
	\caption{Time discretization error of the last time step compared to the number of base time steps for different number of QN time steps.}
	\label{ErrorQNF}
\end{figure}

In Figure \ref{iterationsQNF}, we see that the QNWR method terminated in a finite number of iterations for all test cases, which is in agreement with Theorem \ref{QNWRTATerm}. The number of iterations to reach the termination condition is almost independent of the QNWR time grid.

Looking at the error in the last time step shown in Figure \ref{ErrorQNF}, one can observe that it quickly stagnates for $N_{QN} = 10$. This result is in agreement with Lemma \ref{interpError}, and due to that the QN time grid does not contain enough time steps to accurately represent the solution, resulting in a large interpolation error. 

Similarly, based on the result from Lemma \ref{interpError}, one would expect the error for the case $N_{QN} = 100$ to level off once both solvers do more than a 100 time steps. Additionally, one would expect the error for $N_{QN} = 10000$ to decrease linearly, since the QN time grid is finer than the time grids in the sub solvers. This is exactly what we see in the Air-steel and Air-water test case in Figure \ref{ErrorQNF}, where the case $N_{QN} = 10000$ decreases linearly while the error flattens out for $N_{QN} = 100$ when both solvers do more than a 100 time steps. 

For the Water-steel test case with $N_{QN} = 100$, the error also decreases linearly, which can be explained by that the time integration error is significantly larger than for the Air-steel and Air-water test cases. From these observations, we conclude that there is very little benefit to using more time steps in the QN time grid compared to the number of time steps in the sub-solvers. Furthermore, while using less time steps in the QN time grid reduces the computational cost, it also increases the interpolation error. 

\subsubsection{Comparison of different auxiliary time grids for the time adaptive case}

We now discuss how to choose the auxiliary grid $\mathcal{T}_{QN}$. A simple choice is be to use an equidistant time grid with a user defined number of time steps. However, based on our previous observations we expect that the computationally most efficient strategies use roughly the same number of time steps in the QN time grid as the two sub solvers. We therefore propose to select $\mathcal{T}_{QN}$ automatically based on the sub-solvers time grid after the first iteration. The following three strategies are considered:
 \begin{itemize}
	\item Stategy 1: The timegrid of the Neumann solver in the first iteration, $\mathcal{T}_{N}^1$,
	\item Stategy 2: The timegrid of the Dirichlet solver in the first iteration, $\mathcal{T}_{D}^1$,
	\item Stategy 3:  Equidistant time grid with the minimum number of time steps of the first time grid in the Dirichlet and Neumann solver.
\end{itemize}
To compare these strategies we look at the computational efficiency, error over work, reported in Figure \ref{TATGEfficiency}. The error was measured as the difference between the last time step of the solution and a solution that was computed with a tolerance of $10^{-6}$. The cost of the algorithm can be accurately measured by counting the number of time steps done by the two sub-solvers, since direct solvers are employed inside the two sub-solvers. 

\begin{figure}
	\centering
	\begin{subfigure}[b]{0.3\textwidth}
		\centering
		\includegraphics[width=\textwidth]{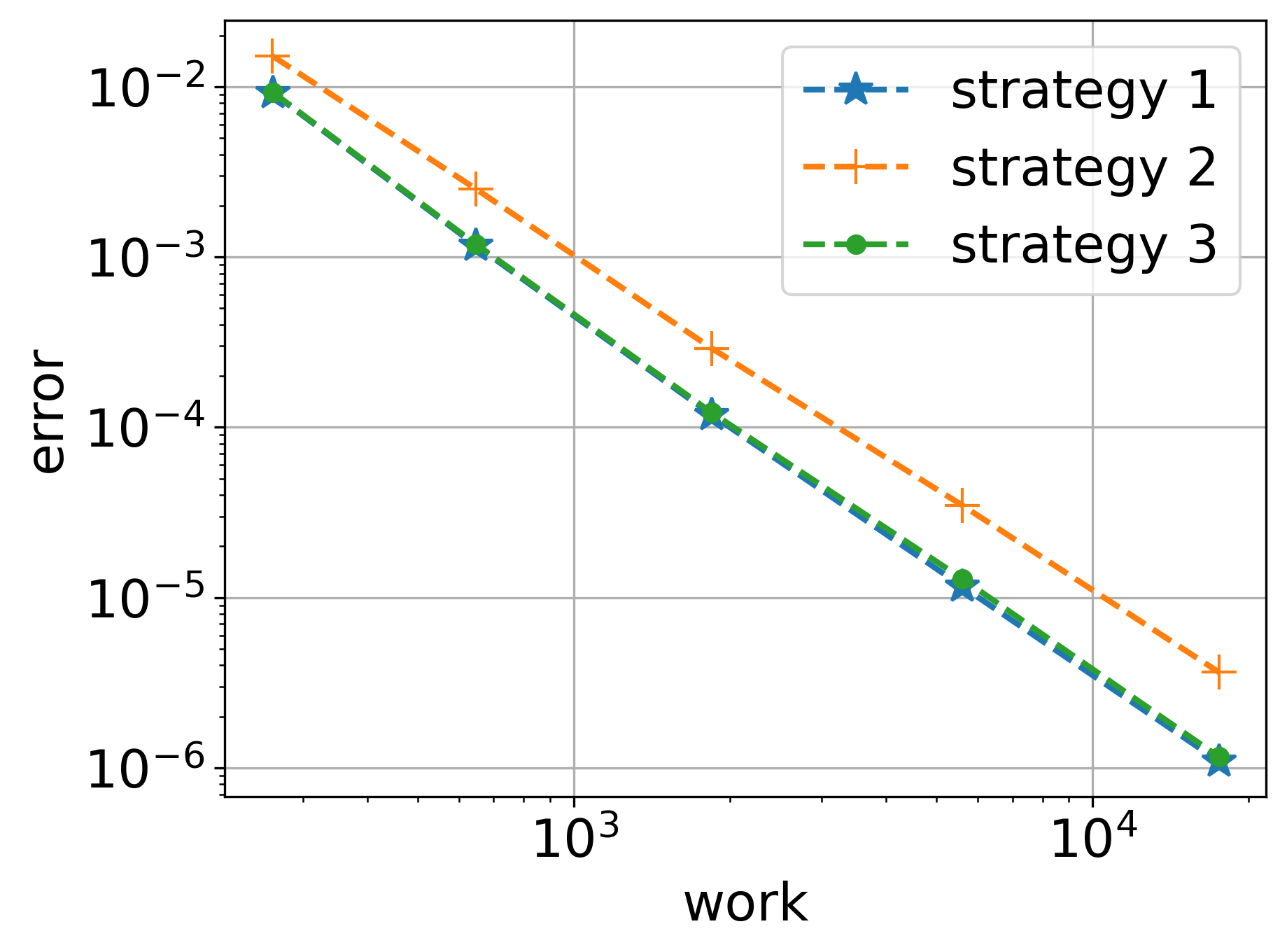}
		\caption{Air-steel}
	\end{subfigure}
	\hfill	
	\begin{subfigure}[b]{0.3\textwidth}
		\centering
		\includegraphics[width=\textwidth]{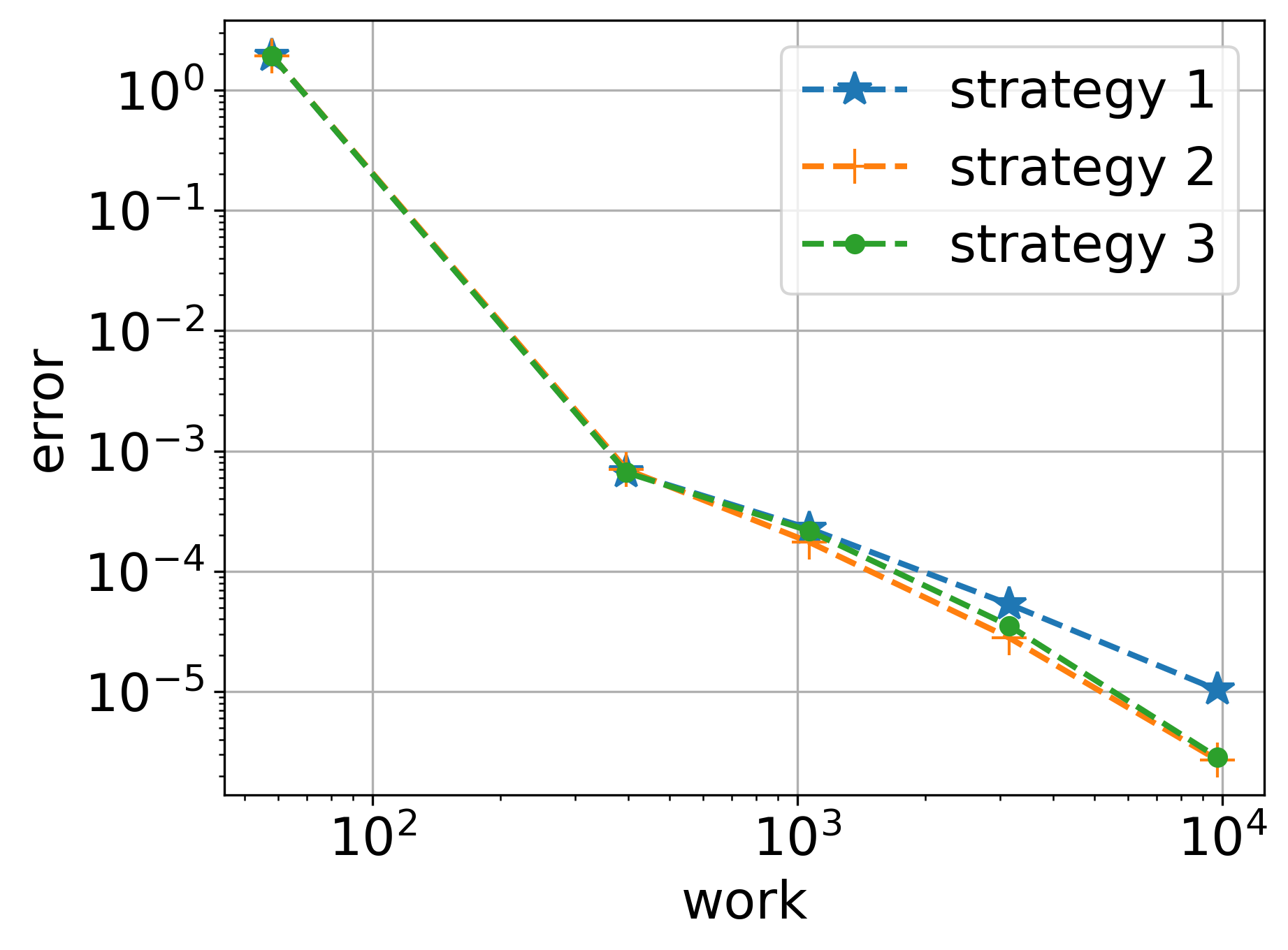}
		\caption{Air-water}
	\end{subfigure}
	\hfill
	\begin{subfigure}[b]{0.3\textwidth}
		\centering
		\includegraphics[width=\textwidth]{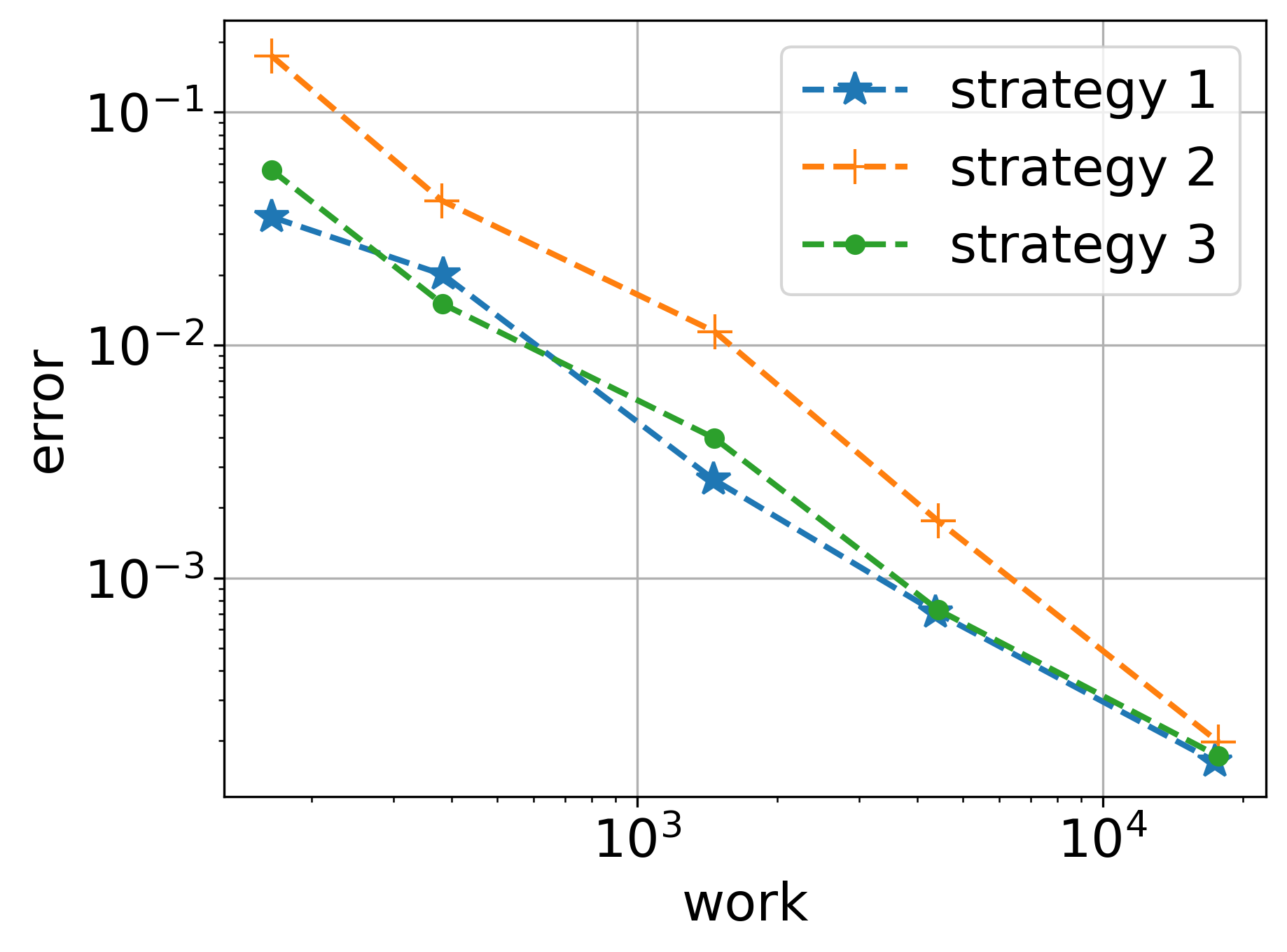}
		\caption{Water-steel}
	\end{subfigure}
	\hfill
	\caption{The computational efficiency of the QNWR algorithm for the three different time grid selection strategies inside the time adaptive QNWR.}
	\label{TATGEfficiency}
\end{figure}

The three different time grid selection strategies in Figure \ref{TATGEfficiency} are quite close with strategy 2 performing slightly worse than strategy 1 and 3. Thus, there is little motivation not to choose strategy 1, since this reduces the QNWR method to \eqref{QNWRFixed} for fixed time steps. It also eliminates the interpolation error in the fixed time step case. 

\subsection{Efficiency of the time adaptive QNWR method}

We now compare the computational efficiency of our time adaptive QNWR algorithm with the optimal relaxation presented in \cite{BiMeMo23} and the multi-rate version of QNWR presented in \cite{RuUe21}. Similar to previous experiments, the code for the sub-solvers is based on \cite{Me25} and can be found here \cite{code}. Unlike in previous experiments we use the QNWR implementation in preCICE that is detailed in section \ref{implementationDetail}. For the optimal relaxation presented in \cite{BiMeMo23}, the coupling code is implemented in a stand alone python code closely following \cite{Me25}. 

We also include preCICE's relaxation feature, which uses a fixed relaxation parameter that is set before the simulation starts. We choose it as the optimal relaxation parameter of the first time step in this case.

As a measure of the computational efficiency we compare error over work, for tolerance $10^{-i}$, where $i = -1,...,-5$. The error was measured as the difference between the solution of the last time step and a solution that was computed with a tolerance of $10^{-6}$. 
For the multirate QNWR we used $N = 8, ..., 256$ base time steps in the sub solvers. Furthermore, a constant tolerance $TOL_{WR}$ of $10^{-6}$ was used in the multirate QNWR, corresponding to the finest tolerance used for the time adaptive methods. Work is measured as the total number of time steps done by the sub solvers, since direct solvers are employed inside the two sub-solvers and the cost of the coupling algorithm is assumed to be negligible compared to the cost of the sub-solvers. The results are presented in Figure \ref{QNEfficiency}.  

\begin{figure}
	\centering
	\begin{subfigure}[b]{0.3\textwidth}
		\centering
		\includegraphics[width=\textwidth]{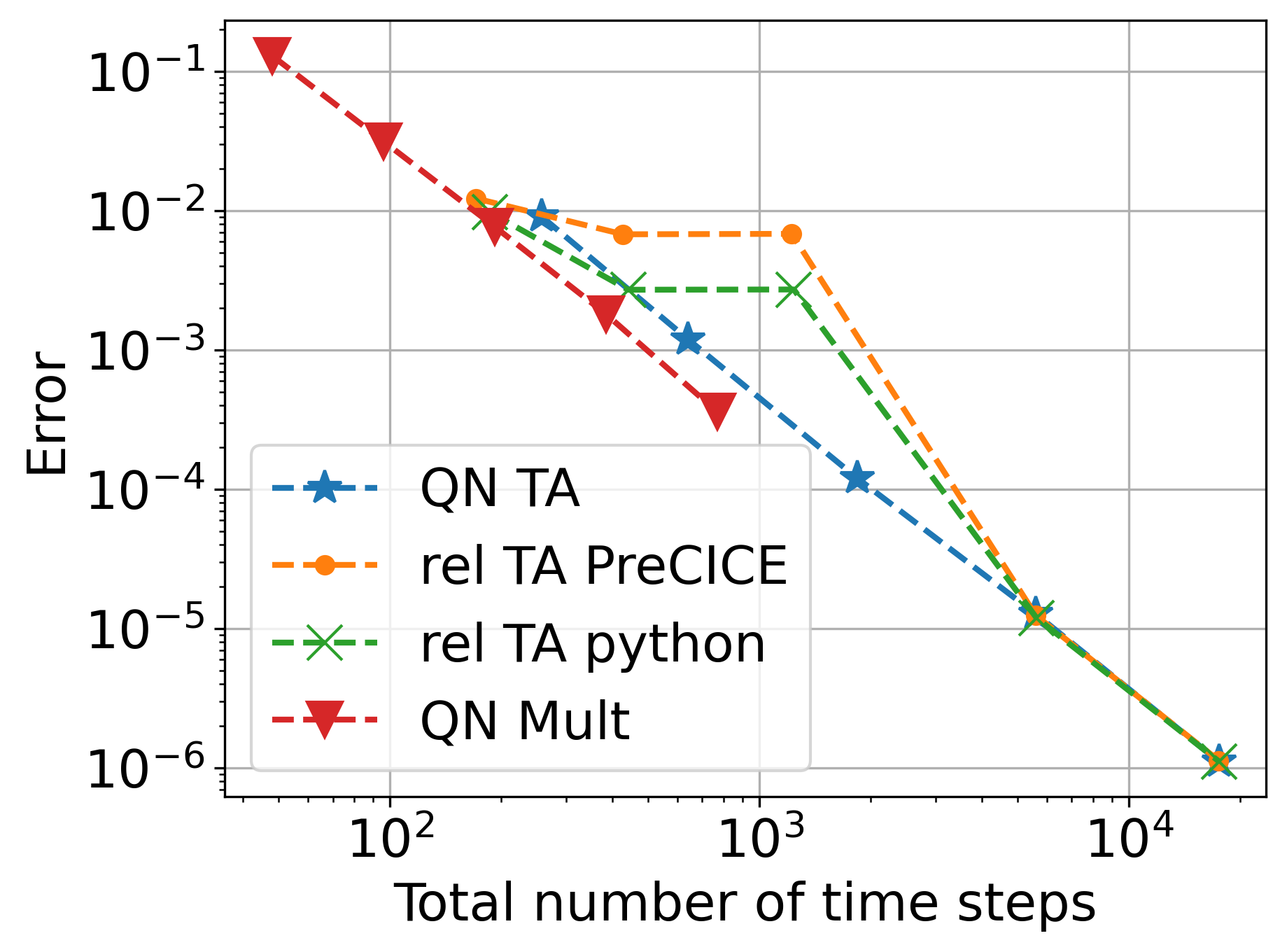}
		\caption{Air-steel}
	\end{subfigure}
	\begin{subfigure}[b]{0.3\textwidth}
		\centering
		\includegraphics[width=\textwidth]{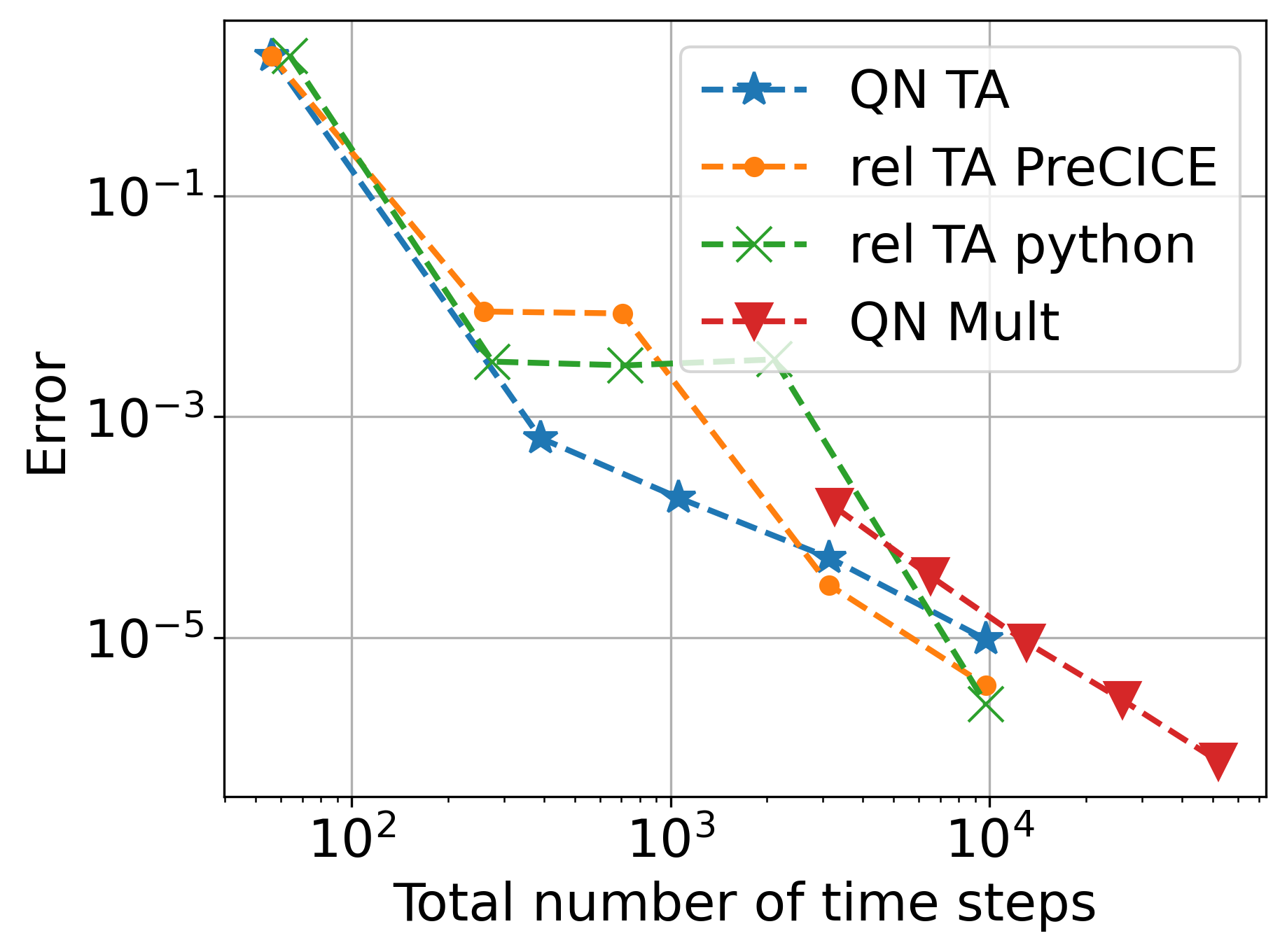}
		\caption{Air-water}
	\end{subfigure}
	\begin{subfigure}[b]{0.3\textwidth}
		\centering
		\includegraphics[width=\textwidth]{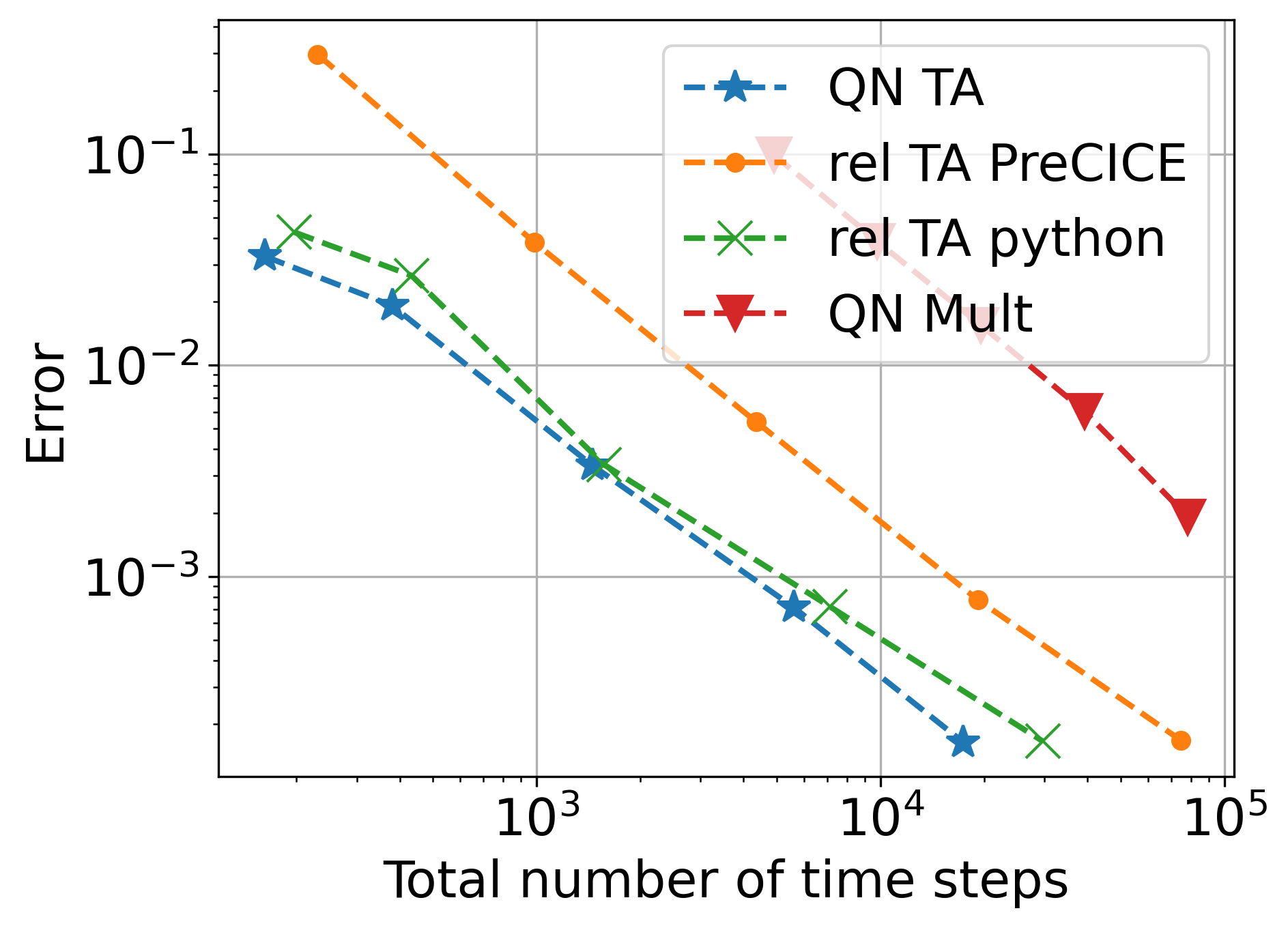}
		\caption{Water-steel}
	\end{subfigure}
	\caption{The error of the last time step compared to different tolerances.}
	\label{QNEfficiency}
\end{figure}

We see that for the Air-steel and Air-water test case all four methods are quite close together, implying that they will give similar performance on these two test cases. However, for the Water-Steel test case, we see a big difference in performance, with time adaptive QNWR being the fastest closely followed by the optimal relaxation parameter from \cite{BiMeMo23}. Using the constant acceleration method in preCICE results in a large loss of performance, which is due to that the optimal relaxation parameter of the first time step differs from the optimal relaxation parameter used in \cite{BiMeMo23}, resulting in significantly more iterations. The constant acceleration feature in preCICE is still more efficient than the multirate QNWR setup even though the multirate QNWR method terminated in fewer iterations. This can in part be explained by that the time adaptive solvers are able to select better step size ratios between the two solvers, making the time adaptive solvers the more robust choice.

%% file: FSITestCase.tex
\section{Fluid-Structure interaction}

We now consider a much more complex setting, were we have different physics, discretization and solvers in our two domains. We use the test case "perpendicular flap" from the preCICE tutorials \cite{ChDa22}. It consists of an elastic beam that is mounted to the lower wall in the middle of a fluid channel, shown in Figure \ref{BeamSimulation}. 

The fluid is modeled by the incompressible Navier-Stokes equations in an arbitrary Lagrangian–Eulerian framework, and is given by
\begin{align*}	
	\frac{\partial \hat{u}}{\partial t} + (\hat{u}-u_g) \cdot \nabla \hat{u} - \nu \Delta \hat{u} &= \frac{\nabla p}{\rho_f}, \quad x \in \Omega_F(t),\\
	\nabla \hat{u} &= 0, \quad x \in \Omega_F(t), 
\end{align*}
where $\hat{u}$ is the velocity of the flow, $u_g$ the velocity of the domain and $p$ the pressure. The fluid domain, denoted by $\Omega_F(t)$, is $6$ m long and $4$ m tall and deforms based on the beams movement, as is demonstrated in Figure \ref{BeamSimulation}. For our test case, the density $\rho_f$ is given as $1 \frac{kg}{m^3}$ and viscosity $\nu$ is set to $1 m^2/s$. Furthermore, zero initial conditions are employed at the start of the simulation in the fluid. To minimize instabilities in the start up phase, the inflow speed on the left boundary is given by \begin{equation*}
\begin{cases}
	[10^4 t, 0] \text{ for } t \leq 10^{-3},\\
	[10, 0] \text{ for } t > 10^{-3}.
\end{cases}
\end{equation*} 
This is combined with a no-slip condition for the walls and the beam, as well as a zero gradient outflow condition on the right. 

The fluid solver uses OpenFOAM's pimpleFoam solver, as in the preCICE tutorial, which uses a finite volume discretization in space and BDF2 in time. Furthermore, the quadrilateral grid from the tutorial, shown in Figure \ref{BeamSimulation}, was also re-used together with openFOAM's Laplacian mesh motion solver. This is combined with OpenFoam's adaptive time stepping strategy, which adjusts the time step based on the CFL number in combination with an initial time step of $10^{-5}$.  

The beam is $0.1$ m wide and $1$ m long and is modeled using linear elasticity:
\begin{align*}
	\rho_s \frac{\partial^2 d}{\partial t^2} &=  \tau + \nabla \cdot \sigma,  \quad x \in \Omega_s,\\
	d &= 0,  \quad x \in \partial \Omega_s \setminus \Gamma.
\end{align*}
Here, $d$ is the displacement, $\tau$ the traction between the fluid and the solid, and $\rho$ the density given by $3 \cdot 10^{3} \frac{kg}{m^3}$ for this test case. Furthermore, $\sigma$ is the Cauchy stress tensor given by
\begin{equation*}
	\sigma = \lambda(\nabla \cdot d) I + \mu\left(\nabla d + (\nabla d)^T \right).
\end{equation*}
The constants $\lambda$ and $\mu$  are given by \begin{align}
	\lambda &= \frac{E \nu }{ ((1.0 + \nu)(1.0 - 2.0\nu))}, \\
	\mu &= \frac{E }{ (2.0(1.0 + \nu))},
\end{align}  where $E$ is the Young's modulus and $\nu$ the Poisson ratio, which for this test case are set to $4 \cdot 10^7 \frac{kg}{m^3}$ and $0.3$ respectively. Lastly, zero initial conditions are employed.

The beam is discretized using linear finite elements in space and SDIRK2 in time. For the spatial discretization, an equidistant triangular grid with a grid size of $10^{-2}$ was employed. Similar to the heat test case, we use an adaptive time stepping strategy based on a local error estimate $l^{n}$. Instead of a PI controller the following dead beat controller, given by
\begin{equation*}
	\Delta t^{k+1, n+1} = \Delta t^{k+1, n} \left(\frac{TOL_{TA}}{||l^{n+1}||_2}\right)^{1/2},
\end{equation*}  
was employed, together with an initial time step of $10^{-5}$. As before, we choose  $TOL_{TA}=TOL_{WR}/5$. The beam solver is implemented using the python bindings of the DUNE-FEM library, which is a powerful open source library for solving partial differential equations \cite{bastian2021}. 

To couple the fluid solver and solid solver, we use the following transmission conditions, given by \begin{align*}
	\frac{\partial d}{\partial t} &= u,  \\
	\tau &= (-p + \nu \left(\nabla u + (\nabla u)^T \right). 
\end{align*}  To solve the coupled problem we use a Dirichlet-Neumann waveform iteration, given by \begin{equation*}
		d^{k+1} = S \circ F(d^k),
\end{equation*}  
where $F$ denotes the fluid solver, which maps the displacements $d$ of the beam to forces on the beams surface. The solid solver is denoted by $S$ and maps the forces from the fluid to the displacement of the beam. 

On the software side, we use preCICE togheter with the OPEN Foam adapter \cite{chourdakis2023} and the python bindings to couple the solid and fluid solver together. The code for the subsolvers and the preCICE configuration file is available under \cite{code}. Furthermore, since we use non matching spatial meshes in the fluid and solid solver, a nearest-neighbor mapping was used to interpolate the data between the solid and the fluid mesh. This allows us to select reasonable spatial grids inside the sub solvers without regarding the spatial grid of the other solver. 

To further simplify the test case, only one time window with a size of $0.2$ was used, which is larger than the maximal desired time step for both the fluid and solid solver. Furthermore, an absolute convergence criterion given as \begin{equation*}
	\sqrt{\Sigma_i r^{n^2}_i \Delta x} \leq TOL_{WR}
\end{equation*} was employed, where $r^n_i$ denotes the residual of the last time step in node $i$.  
 
We now compare the number of iterations for the time adaptive method against no acceleration. To avoid excessive computation times, the maximal number of iterations was set to 10. While this choice is quite strict, diverging simulations tend to use excessively small time steps, resulting in long simulation times. We do three runs with a coarse, medium and fine tolerance of $TOL_{WR} = 5 \cdot 10^{-3}, 5 \cdot 10^{-4}, 5 \cdot 10^{-5}$ and $CFL = 100, 10, 1$ respectively. The resulting number of iterations are shown in Table \ref{beamResults} and the simulation results for the coarsest test case are shown in Figure ~\ref{BeamSimulation}. 

\begin{table}
	\begin{center}
		\begin{tabular}{ |c|c|c| } 
			\hline
			Test-case & time adaptive QNWR & No acceleration \\ \hline
			Coarse & 2 & 2\\ \hline
			Medium & 3 & 3 \\ \hline 
			Fine & 6 & - \\
			\hline
		\end{tabular}
	\end{center}
	\caption{Number of iterations to reach the termination criteria for $T_f = 0.2$. The symbol - denotes that the simulation did not converge within 10 iterations. }
	\label{beamResults}
\end{table} 

\begin{figure}
	\centering
	\begin{subfigure}[b]{0.45\textwidth}
		\centering
		\includegraphics[width=\textwidth]{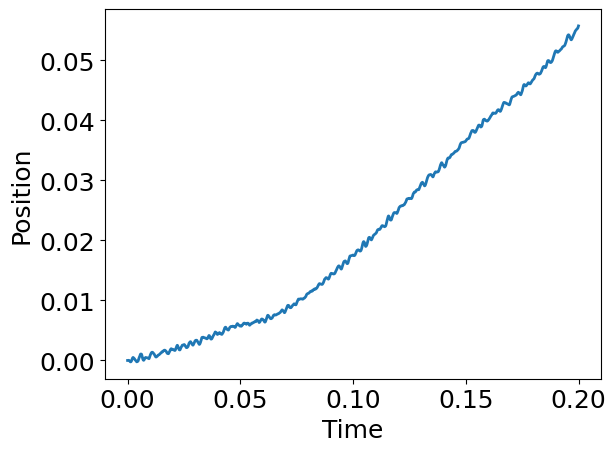}
		\caption{Displacement of the top of the beam during the simulation. }
	\end{subfigure}
	\begin{subfigure}[b]{0.45\textwidth}
		\centering
		\includegraphics[width=\textwidth]{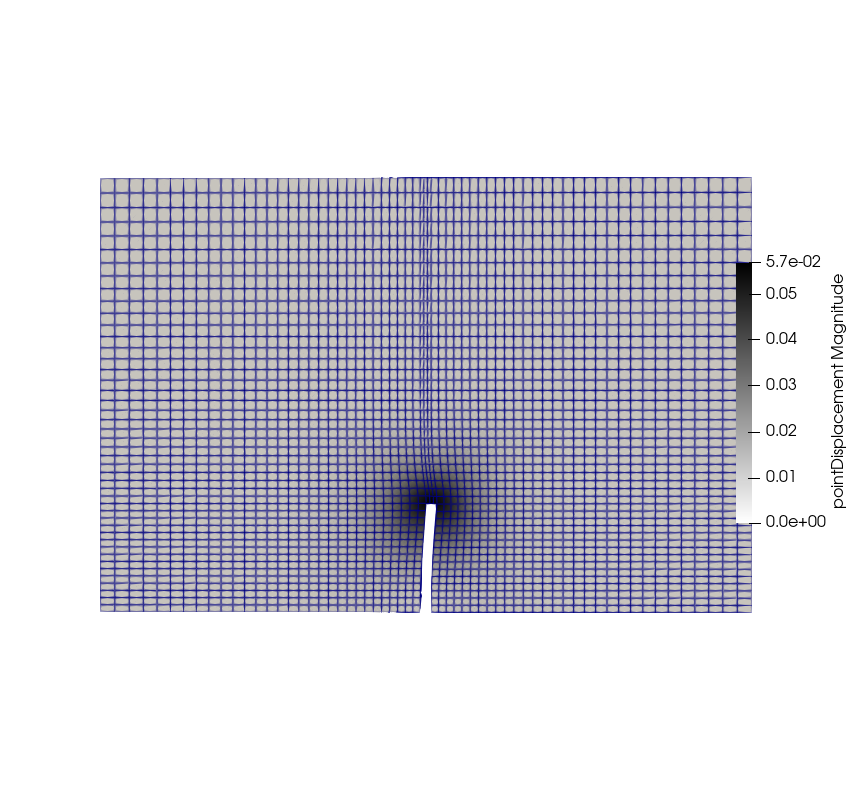}
		\caption{Deformation of the fluid mesh at $0.2$ seconds.}
	\end{subfigure}
	\caption{Simulation results for the coarse test case with $T_f = 0.2$.}
	\label{BeamSimulation}
\end{figure}

From Table \ref{beamResults} we can see that the waveform iteration terminated within three iterations with or without acceleration, corresponding to fast convergence behaviour. However, the fine test case did not terminate in 10 iterations without acceleration, showing that our time adaptive QNWR method increases robustness.